\documentclass[reqno]{amsart}

\usepackage[latin1]{inputenc}
\usepackage{amsmath}
\usepackage{amsfonts}
\usepackage{amssymb}
\usepackage{fullpage}
\usepackage{hyperref}
\usepackage{enumitem}
\usepackage{todonotes}
\usepackage{cleveref}
\usepackage{breqn}
\setkeys{breqn}{breakdepth={1}}

\theoremstyle{plain}
\newtheorem{theorem}{Theorem}
\newtheorem{corollary}[theorem]{Corollary}
\newtheorem{lemma}[theorem]{Lemma}

\theoremstyle{definition}
\newtheorem{definition}[theorem]{Definition}
\newtheorem{example}[theorem]{Example}

\theoremstyle{remark}

\begin{document}
    
    \author{Hieu D. Nguyen}
    \title{A Mixing of Prouhet-Thue-Morse Sequences \\
    and Rademacher Functions}
    \date{5-27-2014}

    \address{Department of Mathematics, Rowan University, Glassboro, NJ 08028.}
    \email{nguyen@rowan.edu}
    
    \subjclass[2010]{Primary 11, Secondary 94}
    
    \maketitle

    \begin{abstract}
        A novel generalization of the Prouhet-Thue-Morse sequence to binary $\pm 1$-weight sequences is presented. Derived from Rademacher functions, these weight sequences are shown to satisfy interesting orthogonality and recurrence relations.  In addition, a result useful in describing these weight sequences as sidelobes of Doppler tolerant waveforms in radar is established.
    \end{abstract} 

\section{Introduction}
Let $v(n)$ denote the binary sum-of-digits residue function, i.e. the sum of the digits in the binary expansion of $n$ modulo 2.  For example, $v(7)=v(111_2)=3$ mod 2 = 1.  Then it is well known that $v(n)$ defines the classical Prouhet-Thue-Morse (PTM) integer sequence, which can easily be shown to satisfy the recurrence
\begin{align*}
v(0) &=0 \\
v(2n) &=v(n) \\
v(2n+1) &=1-v(n)
\end{align*}
The first few terms of $v(n)$ are $0,1,1,0,1,0,0,1$.  Observe that the PTM sequence can also be generated by starting with the value 0 and recursively appending a negated copy of itself (bitwise):
\[0 \rightarrow 01 \rightarrow 0110 \rightarrow 01101001 \rightarrow ...
\]
Another method is to iterate the morphism $\mu$ defined on the alphabet $\{0,1\}$ by the substitution rules $\mu(0)= 01$ and $\mu(1)= 10$ and applied to $x_0=0$ as described in \cite{AS}:
\begin{align*}
x_1 & =\mu(x_0)=01 \\
x_2 & = \mu^2(x_0)=\mu(x_1)=0110 \\
x_3 & = \mu^3(x_0)=\mu(x_2)=01101001 \\
...
\end{align*}

This {\em ubiquitous} sequence, coined as such by Allouche and Shallit in \cite{AS}, first arose in the works of three mathematicians: E. Prouhet involving equal sums of like powers in 1851 (\cite{P}), A. Thue on combinatorics of words in 1906 (\cite{T}), and M. Morse in differential geometry in 1921 (\cite{Mo}).  It has found interesting applications in many areas of mathematics, physics, and engineering: combinatorial game theory (\cite{AS},\cite{P}), fractals (\cite{ASk},\cite{MH}), quasicrystals (\cite{MM},\cite{RSL}) and more recently Dopper tolerant waveforms in radar (\cite{CPH},\cite{PCMH},\cite{NC}).

Suppose we now replace the 0's and 1's in the PTM sequence with 1's and $-1$'s, respectively.  Then it is easy to prove that this yields an equivalent binary $\pm 1$-sequence $w(n)$ satisfying the recurrence
\begin{align*}
w(0) &=1 \\
w(2n) &=w(n) \\
w(2n+1) &=-w(n),
\end{align*}
where $w(n)$ and $v(n)$ are related by
\begin{equation} \label{eq:wv}
w(n)=1-2v(n),
\end{equation}
or equivalently,
\begin{equation} \label{eq:wv2}
w(n)=(-1)^{v(n)}.
\end{equation}

Of course, $v(n)$ can be generalized to any modulus $p\geq 2$.  Towards this end, we define $v_p(n)$ to be the sum of the digits in the base-$p$ expansion of $n$ modulo $p$.  We shall call $v_p(n)$ the mod-$p$ PTM integer sequence.  Then $v_p(n)$ satisfies the recurrence
\begin{align*}
v_p(0) &=0 \\
v_p(pn+r) &=(v(n)+r)_p
\end{align*}
where $(m)_p \equiv m$ mod $p$.  More interestingly, it is well known that $v_p(n)$ provides a solution to the famous Prouhet-Tarry-Escott (PTS) problem (\cite{P},\cite{L},\cite{W}): given a positive integer $M$, find $p$ mutually disjoint sets of non-negative integers $S_0$, $S_1$,...,$S_{p-1}$ so that
\[
\sum_{n\in S_0} n^m = \sum_{n\in S_1} n^m = ... = \sum_{n\in S_{p-1}} n^m
\]
for $m=1,...,M$.  The solution, first given by Prouhet \cite{P} and later proven by Lehmer \cite{L} (see also Wright \cite{W}), is to partition the integers $\{0,1,...,p^{M+1}-1\}$ so that $n \in S_{v_p(n)}$.  For example, if $M=3$ and $p=2$, then 
the two sets $S_0=\{0,3,5,6,9,10,12,15\}$ and $S_1=\{1,2,4,7,8,11,13,14\}$ defined by Prouhet's algorithm solve the PTS problem:
\begin{align*}
60 & = 0+3+5+6+9+10+12+15 \\
 & = 1+2+4+7+8+11+13+14 \\
620 & = 0^2+3^2+5^2+6^2+9^2+10^2+12^2+15^2 \\
& = 1^2+2^2+4^2+7^2+8^2+11^2+13^2+14^2 \\
7200 & = 0^3+3^3+5^3+6^3+9^3+10^3+12^3+15^3 \\
& = 1^3+2^3+4^3+7^3+8^3+11^3+13^3+14^3
\end{align*}

In this paper, we address the following question: what is the natural generalization of $w(n)$ to modulus $p\geq 2$?  Which formula should we look to extend, (\ref{eq:wv}) or (\ref{eq:wv2})?   Is there any intuition behind our generalization?  One answer is to define $w_p(n)$ by merely replacing $v(n)$ with $v_p(n)$ in say  (\ref{eq:wv2}).  However, to discover a more satisfying answer, we consider a modified form of (\ref{eq:wv2}):
\begin{equation} \label{eq:wv-digit}
w(n)=(-1)^{d_{1-v(n)}}
\end{equation}
Here, $d_{1-v(n)}$ takes on one of two possible values, $d_0=1$ or $d_1=0$, which we view as the first two digits in the binary expansion (base 2) of the number 1, i.e. $1=d_12^1+d_02^0$.  Thus, formula (\ref{eq:wv-digit}) involves the digit opposite in position to $v(n)$.  

To explain how this formula naturally generalizes to any positive modulus $p\geq 2$, we begin our story with two arbitrary elements $a_0$ and $a_1$.  Define $A=(a_n)=(a_0,a_1,...)$ to be what we call a mod-2 PTM sequence generated from $a_0$ and $a_1$, where the elements of $A$ satisfy the aperiodic condition
\[
a_n=a_{v(n)}
\]
Thus, $A=(a_0,a_1,a_1,a_0,a_1,a_0,a_0,a_1,...)$.  Since formula (\ref{eq:wv-digit}) holds, it follows that $a_n$ can be decomposed as
\begin{equation} \label{eq:decomposition-mod2}
a_n=\frac{1}{2}(a_0+a_1)+\frac{1}{2}w(n)(a_0-a_1)
\end{equation}
In some sense, $w(n)$ plays the same role as $v(n)$ in defining the sequence $A$, but through the decomposition (\ref{eq:decomposition-mod2}).
We argue that formula (\ref{eq:decomposition-mod2}) leads to a natural generalization of $w(n)$.  For example, suppose $p=3$ and
consider the mod-3 PTM sequence $A=(a_0,a_1,a_2,...)$ generated by three elements $a_0,a_1,a_2$ so that $a_n=a_{v_3(n)}$.  The following decomposition generalizes (\ref{eq:decomposition-mod2}):
\begin{align*}
a_n & =\frac{1}{4}w_0(n)(a_0+a_1+a_2)+\frac{1}{4}w_1(n)(a_0+a_1-a_2) \\
& \ \ \ \ +\frac{1}{4}w_2(n)(a_0-a_1+a_2)+\frac{1}{4}w_3(n)(a_0-a_1-a_2)
\end{align*}
Here, $w_0(n),w_1(n),w_2(n),w_3(n)$ are $\pm 1$-sequences that we shall call the weights of $a_n$. Since $a_n=a_{v_3(n)}$, these weights are fully specified once their values are known for $n=0,1,2$.  It is straightforward to verify in this case that $W(n)=(w_0(n),...,w_3(n))$ takes on the values
\begin{align*}
W(0) & =(1,1,1,1) \\
W(1) & =(1,1,-1,-1) \\
W(2) & =(1,-1,1,-1)
\end{align*}
Thus, the weights $w_i(n)$ are a natural generalization of $w(n)$.  

More generally, if $p\geq 2$ is a positive integer and $A=(a_n)$ is a mod-$p$ PTM sequence generated from $a_0,a_1,...,a_{p-1}$, i.e. $a_n=a_{v_p(n)}$, then the following decomposition holds:
\begin{equation} \label{eq:aB}
a_n=\frac{1}{2^{p-1}} \sum_{i=0}^{2^{p-1}-1} w_i^{(p)}(n) B_i
\end{equation}
Here, the weights $w_i^{(p)}$ are given by
\begin{equation} \label{eq:wi}
w_i(n):=w_i^{(p)}(n)=(-1)^{d^{(i)}_{p-1-v_p(n)}}
\end{equation}
for $0\leq i \leq 2^{p-1}-1$ and $i=d_{p-2}^{(i)}2^{p-2}+...+d_1^{(i)} 2^1+d_0^{(i)}2^0$ is its binary expansion.  Moreover, $B_i$ is calculated by the formula
\begin{equation} \label{eq:B-Rademacher}
B_i=\sum_{n=0}^{p-1}w_i(n) a_{n}
\end{equation}
Observe that we can extend the range for $i$ to $2^p-1$ (and will do so), effectively doubling the number of weights by expanding $i$.
In that case we find that
\[
w_i(n)=-w_{2^p-1-i}(n).
\]
With this extension, we demonstrate in Theorem \ref{th:recurrence}  that each $w_i(n)$ satisfies the recurrence
\[
w_i(pn+r)=w_{x_r(i)}(n)w_i(n)
\]
where $x_r(i)$ denotes a quantity that we define in Section 4 as the `xor-shift' of $i$ by $r$, where $0\leq x_r(i) \leq 2^p-1$.  For example, if $p=2$, we find that
\begin{align*}
w_1(2n) &=w_0(n)w_1(n) \\
w_1(2n+1) &=w_3(n)w_1(n)
\end{align*}
Since $w_0(n)=1$ and $w_3(n)=-1$ for all $n$, this yields the same recurrence satisfied by $w(n)=w_1(n)$ as described in the beginning of this section.

Next, we note that the set of values $R(n)=(w_0(n),...,w_{2^p-1}(n))$ represent those given by the Rademacher functions $\phi_n(x)$, $n=0,1,2,...$, defined by (see \cite{R}, \cite{F})
\begin{align*}
\phi_0(x) & =
1 \ \ (0 \leq x < 1/2) & \phi_0(x+1) & =\phi_0(x)  \\
\phi_0(x) & =-1 \ \ (1/2 \leq x < 1) &
\phi_n(x) & =\phi_0(2^kx)
\end{align*}
In particular,
\[w_i(n)=\phi_n(i/2^p)
\]
so that the right-hand side of (\ref{eq:aB}) can be thought of as a discrete {\em Rademacher} transform of $(B_0,B_1,...,B_{2^{p-1}-1})$.  Moreover, formula (\ref{eq:B-Rademacher}) can be viewed as the inverse transform, which follows from the fact that the Rademacher functions form an orthogonal set.   Thus, weight sequences can be viewed as a mixing of Prouhet-Thue-Morse sequences and Rademacher functions.

It is known that the Rademacher functions generate the Walsh functions, which have important applications in communications and coding theory (\cite{B},\cite{Tz}).  Walsh functions are those of the form (see \cite{F}, \cite{Wa})
\[
\psi_m(x)=\phi_{n_k}(x)\phi_{n_{k-1}}(x)\cdot \cdot \cdot \phi_{n_1}(x)
\]
where $m=2^{n_k} + 2^{n_{k-1}}+...+2^{n_1}$ with $n_i < n_{i+1}$ and $0\leq m \leq 2^p-1$.  This allows us to generalize our weights $w_i(n)$ to sequences
\[
\tilde{w}_i(m)=w_i(n_k)\cdot \cdot \cdot w_i(n_1)
\]
which we view as a discrete version of the Walsh functions in the variable $i$.  In that case, we prove in Section 3 that
\[
\sum_{i=0}^{2^p-1} \tilde{w}_i(m)B_i=\begin{cases}
a_n, & \text{if } m=2^n, 0\leq n \leq p-1 \\
0, & \text{otherwise}
\end{cases}
\]
We also prove in the same section a result that was used in \cite{NC} to characterize these weight sequences as sidelobes of Doppler tolerant radar waveforms (motivated by \cite{CPH} and \cite{PCMH}).

\section{The Prouhet-Thue-Morse Sequence}

Denote by $S(L)$ to be the set consisting of the first $L$ non-negative integers $0,1,...,L-1$.

\begin{definition}
Let $n=n_1n_2...n_k$ be the base-$p$ representation of a non-negative integer $n$.  We define the {\em mod-$p$ sum-of-digits function} $v_p(n) \in \mathbb{Z}_p$ to be the sum of the digits $n_i$ modulo $p$, i.e.
\[
v_p(n)\equiv \sum_{i=1}^k n_i \mod p
\] 
\end{definition}

\noindent Observe that $v_p(n)=n$ if $0\leq n < p$.

\begin{definition}
We define a sequence $A=(a_0,a_1,...)$ to be a {\em mod-$p$ Prouhet-Thue-Morse (PTM)} sequence if it satisfies the aperiodic condition
\[
a_n=a_{v_p(n)}
\]
\end{definition}

\begin{definition} Let $p$ and $M$ be positive integers and set $L=p^{M+1}$.  We define $\{S_0,S_1,...,S_{p-1}\}$ to be a {\em Prouhet-Thue-Morse} (PTM) $p$-{\em block partition} of $S(L)=\{0,1,...,L-1\}$ as follows: if $v_p(n)=i$, then
\[
n\in S_i
\]
\end{definition}

The next theorem solves the famous Prouhet-Tarry-Escott problem.

\begin{theorem}[\cite{P}, \cite{L}, \cite{W}] \label{th:gptm-sequence}
Let $p$ and $M$ be positive integers and set $L=p^{M+1}$.  Suppose $\{S_0,S_1,...,S_{p-1}\}$ is a PTM $p$-block partition of $S(L)=\{0,1,...,L-1\}$.  Then
\[
P_m:=\sum_{n\in S_0} n^m = \sum_{n\in S_1} n^m = ... = \sum_{n\in S_{p-1}} n^m
\]
for $m=1,...,M$.  We shall refer to $P_m$ as the $m$-th Prouhet sum corresponding to $p$ and $M$.
\end{theorem}

\begin{corollary}
Let $A=(a_0,a_1,...,a_{L-1})$ be a mod-$p$ PTM sequence of length $L=p^{M+1}$, where $M$ is a non-negative integer.  Then
\begin{equation}  \label{eq:sum-An}
\sum_{n=0}^{L-1}n^m a_n = P_m (a_0+a_1+...+a_{p-1})
\end{equation}
for $m=0,...,M$.
\end{corollary}

\begin{proof} We have
\begin{align*}
\sum_{n=0}^{L-1}n^m a_n & = \sum_{n\in S_0} n^m a_{v_p(n)} + \sum_{n\in S_1} n^m a_{v_p(n)} + ...+  \sum_{n\in S_{p-1}} n^m a_{v_p(n)} \\
&  = a_0\sum_{n\in S_0} n^m + a_1\sum_{n\in S_1} n^m + ...+ a_{p-1} \sum_{n\in S_{p-1}} n^m \\
& = P_m (a_0+a_1+...+a_{p-1})
\end{align*}
\end{proof}

\section{Weight Sequences}
In this section we develop a generalization of the PTM $\pm1$-sequence $w(n)$ and derive orthogonality and recurrence relations for these generalized sequences that we refer to as {\em weight} sequences.

\begin{definition}
Let $i=d_{p-1}^{(i)}2^{p-1}+d_{p-2}^{(i)}2^{p-2}+...+d_1^{(i)} 2^1+d_0^{(i)}2^0$ be the binary expansion of $i$, where $i$ is a non-negative integer with $0\leq i \leq 2^p-1$.  Define $w_0(n),w_1(n),...w_{2^p-1}(n)$ be binary $\pm 1$-sequences defined by
\[
w_i(n)=(-1)^{d_{p-1-v_p(n)}^{(i)}}
\]
\end{definition}

\begin{example} Let $p=3$.  Then
\begin{align*}
w_0(n) & = (\mathbf{1, 1, 1}, 1, 1, 1, 1, 1, 1,...)  \\
w_1(n) & = (\mathbf{1, 1, -1}, 1, -1, 1, -1, 1, 1,...) \\
w_2(n) & = (\mathbf{1, -1, 1}, -1, 1, 1, 1, 1, -1,...)  \\
w_3(n) & = (\mathbf{1, -1, -1}, -1, -1, 1, -1, 1, -1,...) \\
w_4(n) & = (\mathbf{-1, 1, 1}, 1, 1, -1, 1, -1, 1,...) \\
w_5(n) & =  (\mathbf{-1, 1, -1}, 1, -1, -1, -1, -1, 1,...) \\
w_6(n) & =  (\mathbf{-1, -1, 1}, -1, 1, -1, 1, -1, -1,...) \\
w_7(n) & =  (\mathbf{-1, -1, -1}, -1, -1, -1, -1, -1, -1,...)
\end{align*}
\end{example}
\noindent Observe that the first three values of each weight$w_i(n)$ (displayed in bold) represent the binary value of $i$ if we replace 1 and $-1$ with 0 and 1, respectively.  Morever, we have the following symmetry:

\begin{lemma} \label{le:w-reverse}
For $i=0,1,...,2^p-1$, we have
\[
w_i(n)=-w_{2^p-1-i}(n)
\]
\end{lemma}

\begin{proof} If $i=d_{p-1}^{(i)}2^{p-1}+d_{p-2}^{(i)}2^{p-2}+...+d_0^{(i)}2^0$, then $j=2^p-1-i$ has expansion
\[
j=\bar{d}_{p-1}^{(j)}2^{p-1}+\bar{d}_{p-2}^{(j)}2^{p-2}+...+\bar{d}_0^{(j)}2^0
\]
where $\bar{d}_k^{(j)}=1-d_k^{(i)}$.  It follows that
\[
w_i(n)=(-1)^{d^{(i)}_{p-1-v_p(n)}} = (-1)^{1-d^{(j)}_{p-1-v_p(n)}}=-w_{2^p-1-i}(n)
\]
\end{proof}

\begin{theorem} \label{th:W-orthogonal}
Let $p\geq 2$ be a positive integer.  Then the vectors $W_p(0),W_p(1),..,W_p(p-1)$ defined by
\[
W_p(n)=(w_0^{(p)}(n),w_1^{(p)}(n),...,w_{2^{p-1}-1}^{(p)}(n))
\]
form an orthogonal set, i.e.
\[
W_p(n)\cdot W_p(m) = \sum_{i=0}^{2^{p-1}-1} w_i(n)w_i(m)= 2^{p-1} \delta_{n-m} =
\begin{cases}
2^{p-1}, & n = m \\
0, & n\neq m
\end{cases}
\]
for $0\leq n,m \leq p-1$.  Here, $\delta_n$ is the Kronecker delta function.
\end{theorem}

\begin{proof} It is straightforward to check that the lemma is true for $p=2$.  Thus, we assume $p\geq 3$ and define $k(n)=p-1-n$ so that
\[
W_p(n)\cdot W_p(m)= \sum_{i=0}^{2^{p-1}-1} (-1)^{d^{(i)}_{k(n)}+d^{(i)}_{k(m)}}
\]
Assume $n\neq m $ and without loss of generality, take $n<m$ so that $k(n)>k(m)$.  Assume $0\leq i \leq 2^{p-1}-1$ and expand $i$ in binary so that
\[
i=d^{(i)}_{p-1}2^{p-1}+...+d^{(i)}_{k(n)}2^{k(n)}+...+d^{(i)}_{k(m)}2^{k(m)}+...+d^{(i)}_02^0
\]
where $d_{p-1}^{(i)}=0$.  Suppose in specifying $i$ we fix the choice of values for all binary digits except for $d^{(i)}_{k(n)}$ and $d^{(i)}_{k(m)}$.  Then the set $S=\{(0,0), (0,1),(1,0), (1,1)\}$ consists of the four possibilities for choosing these two remaining digits, which we express as the ordered pair $d=(d^{(i)}_{k(n)},d^{(i)}_{k(m)})$.  But then the contribution from this set of four such values for $i$ sums to zero in the dot product $W_p(n)\cdot W_p(m)$, namely
\[
\sum_{d\in S}(-1)^{d^{(i)}_{k(n)}+d^{(i)}_{k(m)}}=0
\]
Since this holds for all cases in specifying $i$, it follows that $W_p(n)\cdot W_p(m)=0$ as desired.  On the other hand, if $n=m$, then $k(n)=k(m)$ and so $d^{(i)}_{k(n)}=d^{(i)}_{k(m)}$ for all $i$.  It follows that
\[
W_p(n)\cdot W_p(m)=\sum_{i=0}^{2^{p-1}-1} (-1)^{2d^{(i)}_{k(n)}}=\sum_{i=0}^{2^{p-1}-1}1=2^{p-1}
\] 
\end{proof}

In fact, we have the more general result, which states a discrete version of the fact that the Walsh functions form an orthogonal set.

\begin{theorem} \label{th:w-tilde}
Let $m$ be an integer and expand $m=2^{n_k} + 2^{n_{k-1}}+...+2^{n_1}$ in binary with $n_i < n_{i+1}$ and $0\leq m \leq 2^p-1$.  Define
\[
\tilde{w}_i(m)=w_i(n_k)\cdot \cdot \cdot w_i(n_1)
\]
for $i=0,1,...,2^p-1$.  Then
\begin{equation} \label{eq:w-tilde}
\sum_{i=0}^{2^{p-1}-1} \tilde{w}_i(m)=0
\end{equation}
for all $m=0,1,....,2^p-1$.
\end{theorem}

\begin{proof} Let $m=2^{n_k} + 2^{n_{k-1}}+...+2^{n_1}$.  We argue by induction on $k$, i.e. the number of distinct powers of $2$ in the binary expansion of $m$.  Suppose $k=1$ and define $q=p-1-v_p(n_1)$.  Then given any value of $i$ where the binary digit $d_{q}^{(i)}=0$, there exists a corresponding value $j$ whose binary digit $d_{q}^{(j)}=1$.  It follows that
\begin{align*}
\sum_{i=0}^{2^{p-1}-1} \tilde{w}_i(m) & =\sum_{\substack {i=0 \\ d_{q}^{(i)} =0}}^{2^{p-1}-1} (-1)^{d_{q}^{(i)}} +\sum_{\substack {i=0 \\ d_{q}^{(i)} =1}} ^{2^{p-1}-1} (-1)^{d_{q}^{(i)}} \\
& = 2^{p-2}-2^{p-2} = 0
\end{align*}
Next, assume that (\ref{eq:w-tilde}) holds for all $m$ with $k-1$ distinct powers of 2.  Define $q_k=p-1-v_p(n_k)$.  Then for $m$ with $k$ distinct powers of 2, we have
\begin{align*}
\sum_{i=0}^{2^{p-1}-1} \tilde{w}_i(m) & = \sum_{i=0}^{2^{p-1}-1}(-1)^{d_{q_k}^{(i)}+d_{q_{k-1}}^{(i)}+...+d_{q_1}^{(i)}} \\
& = (-1)^0\sum_{\substack {i=0 \\ d_{q_k}^{(i)} =0}}^{2^{p-1}-1} (-1)^{d_{q_{k-1}}^{(i)}+...+d_{q_1}^{(i)}} +(-1)^1 \sum_{\substack {i=0 \\ d_{q_k}^{(i)} =1}} ^{2^{p-1}-1} (-1)^{d_{q_k}^{(i)}+d_{q_{k-1}}^{(i)}+...+d_{q_1}^{(i)}} \\
& = \frac{1}{2}\sum_{i=0}^{2^{p-1}-1} (-1)^{d_{q_{k-1}}^{(i)}+...+d_{q_1}^{(i)}} -\frac{1}{2} \sum_{i=0} ^{2^{p-1}-1} (-1)^{d_{q_k}^{(i)}+d_{q_{k-1}}^{(i)}+...+d_{q_1}^{(i)}} \\
& = \frac{1}{2}\cdot 0 - \frac{1}{2} \cdot 0 = 0
\end{align*}
\end{proof}

In \cite{Ri}, Richman observed that the classical PTM sequence $v(i)$ (although he did not recognize it by name in his paper) can be constructed from the product of all Radamacher functions up to order $p-1$, where $0 \leq i \leq 2^p-1$.  This result easily follows from our formulation of weight sequences since
\begin{align*}
\tilde{w}_i^{(p)}(2^{p}-1) & =w_i^{(p)}(0)w_i^{(p)}(1)\cdot \cdot \cdot w_i^{(p)}(p-1) \\
& = (-1)^{d_{p-1}^{(i)}+d_{p-2}^{(i)}+...+d_{0}^{(i)}} \\
& = v(i)
\end{align*}

Next, we relate weight sequences with PTM sequences.  Since $w_i(n)=-w_{p-1-i}(n)$ from Lemma \ref{le:w-reverse}, the following lemma is immediate.

\begin{lemma} \label{le:weights} Let $A=(a_0,a_1,...)$ be a mod-$p$ PTM sequence.  Define
\[
B_i=\sum_{n=0}^{p-1}w_i(n)a_n
\]
for $i=0,1,...,2^{p}-1$.  Then
\[
B_i(n)=-B_{2^p-1-i}(n)
\]
\end{lemma}

\begin{theorem} The following equation holds for all $n\in \mathbb{N}$:
\begin{equation} \label{eq:weights} 
a_n=\frac{1}{2^{p-1}} \sum_{i=0}^{2^{p-1}-1} w_i(n) B_i
\end{equation}
\end{theorem}

\begin{proof} Since $a_n=a_{v(n)}$ for a PTM sequence, it suffices to prove (\ref{eq:weights}) for $n=0,1,...,p-1$.  It follow from Theorem \ref{th:W-orthogonal} that 
\begin{align*}
\frac{1}{2^{p-1}} \sum_{i=0}^{2^{p-1}-1} w_i(n) B_i 
& = \frac{1}{2^{p-1}} \sum_{i=0}^{2^{p-1}-1} w_i(n) \left( \sum_{m=0}^{p-1}w_i(m)a_m \right) \\
& = \frac{1}{2^{p-1}}\sum_{m=0}^{p-1} \left( \sum_{i=0}^{2^{p-1}-1} w_i(n)w_i(m) \right) a_{m} \\
& = \frac{1}{2^{p-1}}\sum_{m=0}^{p-1} 2^{p-1}\delta_{n-m} a_{m} \\
& = a_n
\end{align*}
\end{proof}

\noindent NOTE: Because of the lemma above, we shall refer to $w_0(n),w_1(n),...,w_{2^{p-1}-1}(n)$ as the PTM weights of $a_n$ with respect to the basis of sums $(B_0,B_1,...,B_{2^{p-1}-1})$.

\begin{example} \

\noindent 1. $p=2$:
\begin{align*}
B_0 & = a_0+a_1 \\
B_1 & = a_0-a_1
\end{align*}

\noindent 2. $p=3$:
\begin{align*}
B_0 & = a_0+a_1+a_2 \\
B_1 & = a_0+a_1-a_2 \\
B_2 & = a_0-a_1+a_2 \\
B_3 & = a_0-a_1-a_2
\end{align*}
\end{example}

\begin{theorem} For $0\leq m \leq 2^p-1$, we have
\begin{equation} \label{eq:w-tilde-B}
\sum_{i=0}^{2^p-1} \tilde{w}_i(m)B_i=\begin{cases}
a_n, & \text{if } m=2^n, 0\leq n \leq p-1 \\
0, & \text{otherwise}
\end{cases}
\end{equation}
\end{theorem}

\begin{proof} If $m=2^n$, then $\tilde{w}_i(n)=w_i(n)$ and thus formula (\ref{eq:w-tilde-B}) reduces to (\ref{eq:weights}).   Therefore, assume $m=2^{n_k}+...+2^{n_1}$ where $k>1$.  Define $S_m=\{0,1,...,p-1\}-\{n_1,n_2,...,n_k\}$.  Then
\begin{align*}
\sum_{i=0}^{2^{p-1}-1} \tilde{w}_i(m) B_i 
& =\sum_{i=0}^{2^{p-1}-1} w_i(n_k)\cdot \cdot \cdot w_i(n_1) \left( \sum_{j=0}^{p-1}w_i(j)a_j \right) \\
& =\sum_{j=0}^{p-1} \left( \sum_{i=0}^{2^{p-1}-1} w_i(n_k)\cdot \cdot \cdot w_i(n_1)w_i(j) \right) a_{j}
\end{align*}
Next, isolate the terms in the outer summation above corresponding to $S_m$:
\begin{align*}
\sum_{i=0}^{2^{p-1}-1} \tilde{w}_i(m) B_i & = a_{n_1} \sum_{i=0}^{2^{p-1}-1} w_i(n_k)\cdot \cdot \cdot w_i(n_2) w_i(n_1)^2+... \\
& \ \ \ \ + a_{n_k} \sum_{i=0}^{2^{p-1}-1} w_i(n_k)^2w_i(n_{k-1})
\cdot \cdot \cdot w_i(n_1) \\
& \ \ \ \ \ \ + \sum_{\substack {j\in S_m}} \left( \sum_{i=0}^{2^{p-1}-1} w_i(n_k)\cdot \cdot \cdot w_i(n_1)w_i(j) \right) a_{k} \\
& = a_{n_1} \sum_{i=0}^{2^{p-1}-1} \tilde{w}_i(m^-_1)+...+ a_{n_k} \sum_{i=0}^{2^{p-1}-1} \tilde{w}_i(m^-_{k}) \\
& \ \ \ \ + \sum_{\substack {j\in S_m}} \left( \sum_{i=0}^{2^{p-1}-1} \tilde{w}_i(m^+_j) \right) a_{k}
\end{align*}
where $m^-_j=m-2^j$ and $m^+_j=m+2^j$.   Now observe that all three summations above with index $i$ must vanish because of Theorem \ref{th:w-tilde}.  Hence,
\[
\sum_{i=0}^{2^{p-1}-1} \tilde{w}_i(m) B_i = 0
\]
as desired.
\end{proof}

We end this section by presenting a result that is useful in characterizing sidelobes of Doppler tolerant waveforms in radar (\cite{PCMH},\cite{CPH},\cite{NC}).

\begin{theorem} \label{th:aBS}
Let $A=(a_0,a_1,...,a_{L-1})$ be a mod-$p$ PTM sequence of length $L=p^{M+1}$, where $M$ is a non-negative integer.  Write
\begin{equation} \label{eq:aBS}
a_n = \frac{1}{2^{p-1}} w_0(n)B_0+\frac{1}{2^{p-1}} S_p(n)
\end{equation}
where 
\[
S_p(n)=\sum_{i=1}^{2^{p-1}-1} w_i(n)B_i
\]
Then
\begin{equation} \label{eq:aBSN}
\sum_{n=0}^{L-1}n^m S_p(n)=N_m(L) B_0
\end{equation}
for $m=1,...,M$ where
\[
N_m(L)=2^{p-1}P_m-\sum_{n=0}^{L-1}n^m
\]
\end{theorem}

\begin{proof} We apply (\ref{eq:sum-An}):
\begin{align*}
\sum_{n=0}^{L-1}n^m S_p(n) & = 2^{p-1} \sum_{n=0}^{L-1}n^m a_n - B_0\sum_{n=0}^{L-1}n^m w_0(n)\\
& = 2^{p-1}P_m (a_0 + a_1+...+a_{p-1}) - B_0 \sum_{n=0}^{L-1}n^m \\
& = (2^{p-1}P_m-\sum_{n=0}^{L-1}n^m)B_0 \\
& = N_m(L) B_0
\end{align*}
\end{proof}

\section{XOR-Shift Recurrence}

In this section we develop a recurrence formula for our weight sequences.  Towards this end, we introduce the notion of an {\em xor-shift} of a binary integer. 

\begin{definition}
Let $a,b\in \mathbb{Z}_2$.  We define $a\oplus b$ to be the exclusive OR (XOR) operation given by the following Boolean truth table:
\begin{align*}
0\oplus 0 & = 0 \\
0\oplus 1 & = 1 \\
1\oplus 0 & = 1 \\
1\oplus 1 & = 0
\end{align*}
More generally, let $x=a_k...a_0$ and $y=b_k...b_0$ be two non-negative integers expressed in binary.  We define $z=x\oplus y=c_k..c_0$ to be the {\em xor bit-sum} of $x$ and $y$, where
\[
c_k=a_k\oplus b_k
\]
\end{definition}

\begin{definition}
We shall say that two integers $a$ and $b$ are congruent modulo 2 and write $a\cong b$ to mean $a = b$ mod 2.
\end{definition}

The following lemma, which is straightforward to prove, will be useful to us.
\begin{lemma} Let $a,b \in \mathbb{Z}$.  Then
\[
a\pm b\cong a\oplus b
\]
\end{lemma}

\begin{definition} Let $p$ be a positive integer and $i$ a non-negative integer with $0\leq i \leq 2^{p}-1$.  Expand $i$ in binary so that
\[
i=d_{p-1}2^{p-1}+...+d_02^0
\]
We define the {\em degree-$p$ xor-shift} of $i$ by $r\geq 0$ to be the decimal value given by the xor bit-sum
\[
x_r(i):=x_r^{(p)}(i)=d_{p-1}...d_rd_{r-1}...d_0 \oplus d_{p-1-r}...d_0d_{p-1}...d_{p-r}
\]
i.e.
\[
x_r(i) = e_{p-1}2^{p-1}+...+e_02^0
\]
where 
\[
e_k=
\begin{cases}
d_k\oplus d_{k-r}, & k\geq r \\
d_k\oplus d_{d+(p-r)}, & k < r 
\end{cases}
\]
for $k=0,1,...,p-1$.
\end{definition}

\begin{example}

\noindent Here are some values of $x_i^{(p)}(n)$ for $p=3$:
\begin{align*}
x_1^{(3)}(0) & = 000_2\oplus 000_2 = 000_2 = 0, & x_2^{(3)}(0) & =000_2\oplus 000_2 = 000_2 = 0 \\
x_1^{(3)}(1) & =001_2\oplus 010_2 = 011_2 = 3, & x_2^{(3)}(1) & =001_2\oplus 100_2 = 101_2 = 5 \\
x_1^{(3)}(2) & =010_2\oplus 100_2 = 110_2 = 6, & x_2^{(3)}(2) & =010_2\oplus 001_2 = 011_2 = 3 \\
x_1^{(3)}(3) & =011_2\oplus 110_2 = 101_2 = 5, & x_2^{(3)}(3) & =011_2\oplus 101_2 = 110_2 = 6
\end{align*}
In fact, when $n=p-1$, the sequence
\[
x_1^{(n+1)}(n) = (0, 3, 6, 5, 12, 15, 10, 9, 24, 27,...)
\]
generates the xor bit-sum of $n$ and $2n$ (sequence A048724 in the Online Encyclopedia of Integer Sequences (OEIS) database: http://oeis.org).
\end{example}

\begin{lemma}
Define 
\begin{align*}
E_p(i,n):= d_{p-1-v_p(n)}^{(i)}
\end{align*}
so that $w_i(n)=(-1)^{E_p(i,n)}$.  Then for $0\leq r < p$, we have
\[
E_p(i,pn+r) =  
\begin{cases}
    d_{p-1-v_p(n)-r}, & \text{if } v_p(n)+r < p\\
    d_{p-1-s}, & \text{if } v_p(n)+r\geq p
    \end{cases}
\]
where $s=v_p(n)+r-p$.  Moreover,
\begin{equation} \label{eq:Ep}
E_p(i,pn+r)-E_p(i,n)\cong E_p(x_r(i),n)
\end{equation}
\end{lemma}

\begin{proof} Since $v_p(pn+r)=(v_p(n)+r)_p$, we have
\[
E_p(i,pn+r) = d_{p-1-(v_p(n)+r)_p}
\]
Now consider two cases: either $v(n)+r < p$ or $v(n)+p\geq p$.  If $v(n)+r < p$, then
\[
E_p(i,pn+r) = d_{p-1-v_p(n)-r}
\]
On the other hand, if $v(n)+r\geq p$, then set $s=v_p(n)+r-p$ so that $(v_p(n)+r)_p=s$.  It follows that
\[
E_p(i,pn+r) = d_{p-1-s}
\]
To prove (\ref{eq:Ep}), we again consider two cases.  First, assume $v_p(n)+r < p$ so that $p-1-v_p(n)\geq r$.  Then
\begin{align*}
E_p(i,pn+r)-E_p(i,n) & = d_{p-1-v_p(n)-r} - d_{p-1-v_p(n)} \\
& \cong d_{p-1-v_p(n)}\oplus d_{p-1-v_p(n)-r} \\
& \cong E_p(x_r(i),n)
\end{align*}
On the other hand, if $v_p(n)+r\geq p$, then set $s=v_p(n)+r-p$ so that $(v_p(n)+r)_p=s$.   Since $p-1-v_p(n)<r$, we have
\begin{align*}
E_p(i,pn+r)-E_p(i,n) & = d_{p-1-s} - d_{p-1-v_p(n)}  \\
& \cong d_{p-1-v_p(n)}\oplus d_{p-1-s} \\
& \cong d_{p-1-v_p(n)}\oplus d_{p-1-v_p(n)+(p-r)} \\
& \cong E_p(x_r(i),n)
\end{align*}
\end{proof}

\begin{theorem} \label{th:recurrence}
Let $p$ be a positive integer.  The weight sequences $w_i(n)$, $0\leq i \leq 2^p-1$, satisfy the recurrence
\begin{equation}
w_i(pn+r)=w_{x_r(i)}(n)w_i(n)
\end{equation}
where $n\in \mathbb{N}$ and $r \in \mathbb{Z}_p$.
\end{theorem}

\begin{proof}  The recurrence follows easily from formula (\ref{eq:Ep}):
\begin{align*}
\frac{w_i(pn+r)}{w_i(n)} & =  (-1)^{E_p(i,pn+r)-E_p(i,n)} \\
& = (-1)^{E_p(x_r(i),n)} \\
&= w_{x_r(i)}(n)
\end{align*}
\end{proof}

\begin{example} Let $p=3$.  Then $w_0(n)=1$ for all $n\in \mathbb{N}$ and the other weight sequences, $w_1(n)$, $w_2(n)$, $w_3(n)$, satisfy the following recurrences:
\begin{align*}
w_1(3n) = w_0(n)w_1(n), \ w_1(3n+1)  = w_3(n)w_1(n), \ w_1(3n+2)  = w_5(n)w_1(n) \\
w_2(3n) = w_0(n)w_2(n), \ w_2(3n+1)  = w_6(n)w_2(n), \ w_2(3n+2)  = w_3(n)w_2(n) \\
w_3(3n) = w_0(n)w_3(n), \ w_3(3n+1)  = w_5(n)w_3(n), \ w_3(3n+2)  = w_6(n)w_3(n)
\end{align*}
\end{example}

\section{Conclusion.}
In this paper we presented what appears to be a novel generalization of the Prouhet-Thue-Morse sequence to weight sequences by considering the Rademacher transform of a given set of elements.  These weight sequences were shown to satisfy interesting recurrences and orthogonality relations.  Moreover, they were used in \cite{NC} to describe sidelobes of Doppler tolerant waveforms to radar.

\vskip 5pt
\noindent  {\em Acknowledgment.}
The authors wish to thank Greg Coxson (Naval Research Laboratory) for many useful discussions on radar and complementary code matrices.


\begin{thebibliography}{1}
    
    \bibitem{AS}
    J.-P. Allouche and J. Shallit, \textit{ The Ubiquitous Prouhet-Thue-Morse Sequence}, Sequences and Their applications, Proc. SETA'98 (Ed. C. Ding, T. Helleseth, and H. Niederreiter). New York: Springer-Verlag, pp. 1-16, 1999.

    \bibitem{ASk}
    J.-P. Allouche and G. Skordev, \textit{ Von Koch and Thue-Morse revisited}, Fractals \textbf{ 15} (2007), no. 4, 405-409.
    
    \bibitem{B}
K. G. Beauchamp, Walsh Functions and Their Applications, Academic Press, London, 1975.

   \bibitem{CPH}
Y. C. Chi, A. Pezeshki, and A. R. Howard, \textit{ Complementary Waveforms for Sidelobe Suppression and Radar Polarimetry}, Principles of Waveform Diversity and Design, M. Wicks, E. Mokole, S. Blunt, R. Schneible and V. Amuso (editors), SciTech Publishing, Raleigh, NC, 2011.

        \bibitem{CH}
        G. E. Coxson and W. Haloupek, {\em Construction of Complementary Code Matrices for Waveform Design}, \textit{ IEEE Transactions on Aerospace and Electronic Systems}, \textbf{49} (2013), No. 3,  1806 - 1816.

   \bibitem{F}
   N. J. Fine, \textit{ On the Walsh functions}, Trans. Amer. Math. Soc. 65 (1949), 372-414.
   
   \bibitem{G}
M. J. E. Golay, \textit{Multislit spectroscopy},  J. Opt. Soc. Am. 39 (1949), 437-444.

   \bibitem{L}
D. H. Lehmer, \textit{ The Tarry-Escott Problem}, Scripta Math., \textbf{ 13} (1947), 37-41.

 \bibitem{MH}
 J. Ma and J. Holdener, \textit{ When Thue-Morse meets Koch}, Fractals \textbf{ 13} (2005), 191-206.

\bibitem{MM}
L. Moretti L and V. Mocella, \textit{ Two-dimensional photonic aperiodic crystals based on Thue-Morse sequence}, Optics Express  \textbf{ 15} (2007), no. 23, 15314-23.

\bibitem{Mo}
M. Morse, \textit{ Recurrent Geodesics on a Surface of Negative Curvature}, Trans. Amer. Math. Soc. \textbf{ 22} (1921), 84-100.

   \bibitem{NC}
   H. D. Nguyen and G. E. Coxson, \textit{ Doppler Tolerance, Complementary Code Sets, and the Generalized Thue-Morse Sequence}, preprint, 2014.
   
   \bibitem{P}
   I. Palacios-Huerta, \textit{ Tournaments, fairness and the prouhet-thue-morse sequence}, Economic inquiry \textbf{ 50} (12012), no. 3, 848-849. 

\bibitem{PCMH}
A. Pezeshki, A. R. Calderbank, W. Moran, and S. D. Howard,  
\textit{ Doppler Resilient Golay Complementary Waveforms}, 
IEEE Transactions on Information Theory \textbf{ 54} (2008), no. 9, 4254 - 4266.

        \bibitem{P}
        E. Prouhet, \textit{ Memoire sur Quelques Relations Entre les Puis- sances des Nombres}, C. R. Acad. Sci., Paris, \textbf{ 33} (1851), 225.

\bibitem{R}
H. Rademacher, \textit{ Einige S{\" a}tze {\" u}ber Reihen von allgemeinen Orthogonalfunktionen}, Math. Ann. 87 (1922), no. 1-2, 112-138 (German)

\bibitem{Ri}
R. M. Richman, \textit{ Recursive Binary Sequences of Differences}, Complex Systems 13 (2001) 381-392. 

\bibitem{RSL}
R. Riklund, M. Severin, and Y. Liu, \textit{ The Thue-Morse aperiodic crystal: a link between the Fibonacci quasicrytal and the periodic crystal}, Int. J. Mod. Phys. B \textbf{ 1} (1987), no. 1, 121-132.

\bibitem{Tz}
S. G. Tzafestas, Walsh Functions in Signal and Systems Analysis and Design (Benchmark Papers in Electrical Engineering and Computer Science, Vol 31), Springer, 1985.

\bibitem{T}
A. Thue, \textit{ Uber unendliche Zeichenreihen}, Kra. Vidensk. Selsk. Skrifter. I. Mat.-Nat., Christiana, Nr. 10. 1912 (Reprinted in Selected Mathematical Papers of Axel Thue, edited by T. Nagell. Oslo: Universitetsforlaget, 1977, 139-58).

\bibitem{Wa}
J. L. Walsh, \textit{ A Closed Set of Normal Orthogonal Functions}, Amer. J. Math. \textbf{45} (1923), no. 1, 5-24.

        \bibitem{W}
        E. M. Wright, \textit{ Prouhet's 1851 Solution of the Tarry-Escott Problem of 1910}, Amer. Math. Monthly \textbf{102} (1959), 199-210.

        
    \end{thebibliography}
\end{document}